%% file: 2geneSTABLE.tex
\documentclass[12pt]{amsart}
\usepackage{amsfonts,amssymb,amscd,amsmath,enumerate,verbatim,color}
\usepackage[latin1]{inputenc}
\usepackage{amscd}
\usepackage{latexsym}

\usepackage{graphicx} %======> for arxiv!!

\usepackage{mathptmx}
\def\NZQ{\mathbb}               % the font for N,Z,Q,R,C

\def\ZZ{{\NZQ Z}}
\def\RR{{\NZQ R}}

\def\ab{{\mathbf a}}

\def\eb{{\mathbf e}}
\def\tb{{\mathbf t}}

\def\xb{{\mathbf x}}

\def\opn#1#2{\def#1{\operatorname{#2}}} % to make operators
\opn\gr{gr}

\def\Ac{{\mathcal A}}

\def\Mc{{\mathcal M}}

\def\Gc{{\mathcal G}}

\def\Qc{{\mathcal Q}}

\newtheorem{Theorem}{Theorem}[section]

\newtheorem{Corollary}[Theorem]{Corollary}
\newtheorem{Proposition}[Theorem]{Proposition}

\theoremstyle{definition}
\newtheorem{Remark}[Theorem]{Remark}

\newtheorem{Conjecture}[Theorem]{Conjecture}

\newtheorem*{acknowledgement}{Acknowledgment}
\let\epsilon\varepsilon
\let\phi=\varphi
\let\kappa=\varkappa
\opn\dis{dis}
\opn\height{height}
\opn\dist{dist}
\def\pnt{{\raise0.5mm\hbox{\large\bf.}}}

\opn\Lex{Lex}
\opn\conv{conv}

\begin{document}

\title{Perfectly contractile graphs and quadratic toric rings}

\author{Hidefumi Ohsugi, Kazuki Shibata and Akiyoshi Tsuchiya}

\address{Hidefumi Ohsugi,
	Department of Mathematical Sciences,
	School of Science,
	Kwansei Gakuin University,
	Sanda, Hyogo 669-1337, Japan} 
\email{ohsugi@kwansei.ac.jp}

\address{Kazuki Shibata,
Department of Mathematics,
College of Science,
Rikkyo University,
Toshima-ku, Tokyo 171-8501, Japan} 
\email{k-shibata@rikkyo.ac.jp }

\address{Akiyoshi Tsuchiya,
Graduate School of Mathematical Sciences,
University of Tokyo,
Komaba, Meguro-ku, Tokyo 153-8914, Japan} 
\email{akiyoshi@ms.u-tokyo.ac.jp}

\subjclass[2010]{}

\begin{abstract}
Perfect graphs form one of the distinguished classes of finite simple graphs.
In 2006, Chudnovsky, Robertson, Seymour and Thomas proved that a graph is perfect if and only if it has no odd holes and no odd antiholes as induced subgraphs, which was conjectured by Berge.
We consider the class $\Ac$ of graphs that have no odd holes, no antiholes and no odd stretchers as induced subgraphs. In particular, every graph belonging to $\Ac$ is perfect.
Everett and Reed conjectured that a graph belongs to $\Ac$ if and only if it is perfectly contractile.
In the present paper, we discuss graphs belonging to $\Ac$ from a viewpoint of commutative algebra.
In fact, we conjecture that a perfect graph $G$ belongs to $\Ac$ if and only if the toric ideal of the stable set polytope of $G$ is generated by quadratic binomials.
Especially, we show that this conjecture is true for Meyniel graphs,
perfectly orderable graphs, and clique separable graphs, which are 
perfectly contractile graphs.
\end{abstract}

\maketitle

\input{intro.tex}

\begin{acknowledgement}
The authors were partially supported by JSPS KAKENHI 18H01134, 19K14505 and 19J00312.
\end{acknowledgement}

\section{A necessary condition for quadratic generation}

In the present section, we will show the ``only if'' part of Conjecture \ref{toricconjecture}.
First, we introduce fundamental tools for studying
a set of generators and Gr\"obner bases of $I_G$.

\begin{Remark}
(a)
Since the stable set polytope $\Qc_G$ of $G$ is a (0,1)-polytope, 
the initial ideal of $I_G$ is squarefree if it is quadratic.
See \cite[Proposition 4.27]{BinomialIdeals}.

\noindent
(b)
It is easy to see that $I_G =\{0\}$ if and only if $G$ is a complete graph.
Even if $I_G =\{0\}$, we say that ``$I_G$ is generated by quadratic binomials''
and ``$I_G$ has a quadratic Gr\"obner basis''.
\end{Remark}

It is known (e.g., \cite{MOS}) that quadratic generation is a hereditary property.

\begin{Proposition}
\label{induced}
Let $G'$ be an induced subgraph of a graph $G$.
Then we have the following{\rm :}
\begin{itemize}
\item[(a)]
If $I_G$ is generated by binomials of degree $\le r$,
then so is $I_{G'}${\rm ;}
\item[(b)]
If $I_G$ has a Gr\"obner basis consisting of binomials of degree $\le r$,
then so does $I_{G'}${\rm ;}
\item[(c)]
If $I_G$ has a squarefree initial ideal,
then so does $I_{G'}$.
\end{itemize}
\end{Proposition}

By the following fact (\cite[Example~2]{MOS}), we may assume that $G$ is connected.

\begin{Proposition}
Let $G_1, \dots, G_s$ denote the connected components
of a graph $G$.
Then the toric ring $K[G]$ is the Segre product of $K[G_1],
\dots, K[G_s]$. In particular, 
$I_G$ is generated by quadratic binomials
(resp.~has a squarefree quadratic initial ideal) if and only if 
all of $I_{G_1},\dots, I_{G_s}$ are generated by quadratic binomials
(resp.~have a squarefree quadratic initial ideal).
\end{Proposition}

A {\em clique} in a graph is a set of pairwise adjacent vertices in the graph.
Let $G$ be a connected graph on the vertex set $[n]$.
A subset $C \subset [n]$ is called a {\em cutset} of $G$
 if the induced subgraph of $G$ on the vertex set $[n] \setminus C$
is not connected.
We say that a graph $G$ is obtained by gluing graphs $H_1$ and $H_2$
along their common clique if $H_1 \cap H_2$ is a complete graph
and $G = H_1 \cup H_2$.

\begin{Proposition}[{\cite{EnNo}}]
\label{clique_sum}
Let $G$ be a connected graph with a clique cutset $C$.
Suppose that $G$ is obtained by gluing two graphs 
$H_1$ and $H_2$ along their common clique on $C$.
Then $I_G$ is generated by quadratic binomials
(resp.~has a squarefree quadratic initial ideal)
if and only if $I_{H_1}$ and $I_{H_2}$ are 
generated by quadratic binomials
(resp.~have a squarefree quadratic initial ideal).
\end{Proposition}

The following proposition says that $I_G$ rarely has a quadratic  Gr\"obner basis
if $G$ is not perfect.

\begin{Proposition}[{\cite[Proposition~1.3]{matsuda2018}}]
\label{matsuda_antihole}
Suppose that $H$ is an odd antihole with $\ge 7$ vertices.
Then $I_H$ is generated by quadratic binomials and has
no quadratic  Gr\"obner basis.
In particular, if a graph $G$ has an odd  antihole with $\ge 7$ vertices,
then $I_G$ has no quadratic  Gr\"obner basis.
\end{Proposition}

The following necessary condition is known.

\begin{Proposition}[{\cite[Proposition~11]{MOS}}]
\label{even_antihole}
If $I_G$ is generated by quadratic binomials,
then $G$ has no even antiholes.
\end{Proposition}

We now give a new necessary condition.

\begin{Theorem}
\label{hitsuyou}
If $I_G$ is generated by quadratic binomials,
then $G$ has no odd stretchers as induced subgraphs.
\end{Theorem}

\begin{proof}
By Proposition~\ref{induced}, 
it is enough to show that the toric ideal $I_{G_{s,t,u}}$ of an odd stretcher $G_{s,t,u}$ is not generated by quadratic binomials.
It is easy to see that the stability number $\max\{ |S| : S \in S(G_{s,t,u})  \}$ of $G_{s,t,u}$ is $s+t+u-1$.
Let $S(G_{s,t,u}) = \{S_1,\ldots,S_m\}$, where 
%$S_1, \dots, S_6$ be the following stable sets of $G_{s,t,u}$:
\begin{eqnarray*}
S_1 &=& \{i_1, i_3, \dots, i_{2s-1}, j_2, j_4, \dots, j_{2t}, k_2, k_4, \dots, k_{2u -2}\},\\
S_2 &=& \{i_2, i_4, \dots, i_{2s-2}, j_1, j_3, \dots, j_{2t-1}, k_2, k_4, \dots, k_{2u }\},\\
S_3 &=& \{i_2, i_4, \dots, i_{2s}, j_2, j_4, \dots, j_{2t-2}, k_1, k_3, \dots, k_{2u -1}\},\\
S_4 &=& \{i_1, i_3, \dots, i_{2s-1}, j_2, j_4, \dots, j_{2t-2}, k_2, k_4, \dots, k_{2u}\},\\
S_5 &=& \{i_2, i_4, \dots, i_{2s}, j_1, j_3, \dots, j_{2t-1}, k_2, k_4, \dots, k_{2u -2}\},\\
S_6 &=& \{i_2, i_4, \dots, i_{2s-2}, j_2, j_4, \dots, j_{2t}, k_1, k_3, \dots, k_{2u -1}\}.
\end{eqnarray*}
Note that $|S_i| = s+t+u-1$ for $i=1,2,\dots,6$.
Since $\rho(S_1) + \rho(S_2) + \rho(S_3) = \rho(S_4) + \rho(S_5) + \rho(S_6)$,
the binomial $f=x_1 x_2 x_3 - x_4 x_5 x_6$ belongs to $I_{G_{s,t,u}}$.
Suppose that $I_{G_{s,t,u}}$ is generated by quadratic binomials.
Then there exist $1 \le i < j \le 3$ and stable sets
$S_k, S_\ell$ of $G_{s,t,u}$ with $|S_k| =|S_\ell| = s+t+u-1$
such that $x_i x_j - x_k x_\ell$ is a nonzero binomial in $I_{G_{s,t,u}}$.
By the symmetries on $S_1, S_2, S_3$, we may assume that $i=1$ and $j=2$.
Since $x_i x_j - x_k x_\ell$ belongs to $I_{G_{s,t,u}}$,
we have
$$S_1 \cap S_2 = S_k \cap S_\ell = \{k_2, k_4, \dots, k_{2u -2}\},$$
$$(S_k \cup S_\ell) \setminus (S_k \cap S_\ell ) = 
\{i_1, i_2, \dots, i_{2s-1}, j_1, j_2, \dots, j_{2t}, k_{2u}\}.
$$
Since there is exactly one way
$$
\{i_1, i_3, \dots, i_{2s-1}, j_2, j_4, \dots, j_{2t}\}
\cup
\{i_2, i_4, \dots, i_{2s-2}, j_1, j_3, \dots, j_{2t-1}, k_{2u}\}
$$
to decompose $(S_k \cup S_\ell) \setminus (S_k \cap S_\ell ) $
into two stable sets,
it follows that $\{S_1,S_2\} = \{S_k, S_\ell\}$,
a contradiction.
\end{proof}

By Proposition \ref{even_antihole} and Theorem \ref{hitsuyou},
the ``only if'' part of Conjecture \ref{toricconjecture} is true.

\section{Quadratic generation for Meyniel graphs}

A graph is called {\em Meyniel} if any odd cycle of length $\ge 5$ has at least two chords.
In the present section, we will prove that $I_G$ is generated by quadratic binomials
if $G$ is Meyniel.
Meyniel graphs are one of the important classes of perfect graphs.
Several characterization of Meyniel graphs are known.
For example, it is known \cite{Hoang} that
a graph $G$ is Meyniel if and only if $G$ is {\em very strongly perfect},
i.e., for every induced subgraph $H$ of $G$, 
every vertex of $H$ belongs to a stable set of $H$ meeting all maximal cliques of $H$. 
Hertz \cite{MeynielColor} introduced the following algorithm:

\bigskip

\noindent
\underline{{\bf COLOR} \textit{with rule ${\mathcal R}$}}:

{\bf Input}: a Meyniel graph $G$

{\bf Output}: a coloring of the vertices of $G$

1. $G_0 := G$; $k:=0$; $(vw)_0$ is any vertex of $G$;

2. While $G_k$ is not a clique do:

\ \ \ \ \ \ 2.1. Choose two non-adjacent vertices $v_k$ and $w_k$
by the following rule (``rule ${\mathcal R}$'');
\begin{itemize}
\item
If there exists at least one vertex not adjacent to $(vw)_k$,
then $v_k:=(vw)_k$,
else choose for $v_k$ any vertex not adjacent to every other vertex in $G_k$;
\item
Choose for $w_k$ any vertex not adjacent to $v_k$ such that
the number of common neighbors with $v_k$ is maximal
among the vertices not adjacent to $v_k$. 
\end{itemize}

\ \ \ \ \ \ 2.2. Construct $G_{k+1}$ by contracting $v_k$ and $w_k$
into a vertex $(vw)_{k+1}$;

\ \ \ \ \ \ 2.3. $k:=k+1$;

3. Color the clique $G_k$;

4. While $k\neq 0$ do:

\ \ \ \ \ \ 4.1. $k:=k-1$;

\ \ \ \ \ \ 4.2. Decontract $G_{k+1}$ by giving to $v_k$ and $w_k$ the same color
as $(vw)_{k+1}$.

\bigskip

Then the following is a part of the results given by Hertz \cite{MeynielColor}.

\begin{Proposition}
\label{meyniel_even_pairs}
If we input a Meyniel graph $G$ to the algorithm {\bf COLOR} with rule ${\mathcal R}$,
then we have the following{\rm :}
\begin{itemize}
\item[(a)]
Each $\{v_r, w_r\}$ is an even pair in $G_r${\rm ;}
\item[(b)]
Let $s$ be the smallest index such that $v_s = (vw)_s$ and 
either $G_{s+1}$ is a clique or else $v_{s+1} \neq (vw)_{s+1}$.
Then the set $S=\{v_0,w_0,w_1,\dots,w_s\}$ is a stable set of $G$ meeting all maximal cliques of $G$. 
\end{itemize}
\end{Proposition}

Using Proposition~\ref{meyniel_even_pairs}, we have the following.

\begin{Theorem}
\label{thm:meyniel}
Let $G$ be a Meyniel graph.
Then $I_G$ is generated by quadratic binomials.
\end{Theorem}

\begin{proof}
Let $f= x_{i_1} \dots x_{i_r} - x_{j_1} \dots x_{j_r} \in I_G$ be a binomial of degree $r \ge 3$
such that $f$ is not generated by binomials in $I_G$ of degree $\le r-1$.
Note that $i_1,\ldots,i_r$ and $j_1, \ldots,j_r$ are not necessarily distinct.
Since $I_G$ is prime, $f$ must be irreducible.
We consider the induced subgraph $H$ of $G$ on the vertex set $\bigcup_{\ell=1}^r S_{i_\ell}$
$(=\bigcup_{\ell=1}^r S_{j_\ell})$.
Since $G$ is Meyniel, so is $H$.
Let $S=\{v_0,w_0,w_1,\dots, w_s\}$ be a stable set of $H$ meeting all maximal cliques of $H$
obtained by the algorithm {\bf COLOR} with rule ${\mathcal R}$.
Note that
\begin{equation*}
S \subset \bigcup_{\ell=1}^r S_{i_\ell} =\bigcup_{\ell=1}^r S_{j_\ell}
\end{equation*}
since $H$ is the induced subgraph of $G$ on $\bigcup_{\ell=1}^r S_{i_\ell}$.
We may assume that $v_0$ belongs to $S_{i_1}$.
Suppose that $S \not\subset S_{i_1}$.
Let $u = \min\{ \mu : w_\mu \notin S_{i_1}\}$.
Then we may assume that $w_u \in S_{i_2}$.
Let $H'$ be the induced subgraph of $G$ on the vertex set $S_{i_1} \cup S_{i_2}$.
Since both $S_{i_1}$ and $S_{i_2}$ are stable sets of $G$, it follows that 
$H'$ is the disjoint union of graphs $G_1$ and $G_2$, where
$G_1$ is a bipartite graph on the vertex set $(S_{i_1} \setminus S_{i_2}) \cup (S_{i_2} \setminus S_{i_1})$,
 and $G_2$ is an empty graph on the vertex set $S_{i_1} \cap S_{i_2}$ ($G_2$ is not needed when $S_{i_1} \cap S_{i_2} = \emptyset$).
Let $H_1,\dots, H_t$ be the connected components of $G_1$.
Since each $H_k$ is a connected bipartite graph, the vertex set of
$H_k$ has the unique bipartition $V_1^{(k)} \cup V_2^{(k)}$
with $V_1^{(k)} \subset S_{i_1} \setminus S_{i_2}$ and $V_2^{(k)}\subset S_{i_2} \setminus S_{i_1}$.
Then there exists $ 1\le \alpha \le t$ such that $w_u$ belongs to $ V_2^{(\alpha)}$.
Suppose that $p \in \{v_0, w_0,\dots, w_{u-1}\} \ (\subset  S_{i_1})$
belongs to $H_\alpha$.
Let $H''$ be the bipartite graph obtained by contracting the vertices $\{v_0, w_0,\dots, w_{u-1}\}$ 
into $v_u$ in $H'$.
Then $v_u$ and $w_u$ belong to the same connected component of $H''$,
and hence there exists an induced odd path from $v_u$ to $w_u$ in $H''$.
Since $H''$ is an induced subgraph of $G_u$, this contradicts that $\{v_u ,w_u\}$ is an even pair in $G_u$.
Hence $p \notin H_\alpha$.
It then follows that 
\begin{eqnarray*}
S_{i_1'} &=& \left(S_{i_1} \setminus V_1^{(\alpha)} \right) \cup V_2^{(\alpha)}\\
S_{i_2'} &=& \left(S_{i_2} \setminus V_2^{(\alpha)} \right) \cup V_1^{(\alpha)}
\end{eqnarray*}
are stable sets of $G$
satisfying $\{v_0, w_0,\dots, w_{u}\} \subset S_{i_1'}$ and $x_{i_1} x_{i_2} -  x_{i_1'} x_{i_2'} \in I_G$.
Then 
$$
f
=
x_{i_3} \dots x_{i_r}  (x_{i_1} x_{i_2} -  x_{i_1'} x_{i_2'} )
+ f',
$$ 
where $f'=  x_{i_1'} x_{i_2'}  x_{i_3} \dots x_{i_r} - x_{j_1} \dots x_{j_r}
\in I_G$.
We may replace $f$ with $f'$.
Applying this procedure repeatedly, 
we may assume that $S$ is a subset of $S_{i_1}$.
Since $S$ meets all maximal cliques of $H$, each vertex of $H \setminus S$ is adjacent to a vertex in $S$. 
Thus $S_{i_1} = S$.
Applying the same procedure to $ x_{j_1} \dots x_{j_r}$, we may also assume that
$S_{j_1} = S$.
This contradicts that $f$ is irreducible. 
\end{proof}

\section{Quadratic Gr\"obner bases for perfectly orderable graphs}

Let $G$ be a graph on the vertex set $\{v_1,\dots,v_n\}$.
An ordering $(v_1,\dots,v_n)$ of the vertex set of $G$ is called {\em perfect} \cite[Definition~5.6.1]{graphclasses}
if, for any induced (ordered) subgraph $H$ of $G$,
the number of colors used by
Greedy Coloring \cite[Algorithm~5.6.1]{graphclasses} on $H$ 
coincides with the chromatic number of $H$.
It is known \cite[Theorem~1]{perfectlyorderable} that
a vertex ordering $<$ of a graph $G$ is perfect if and only if $G$ contains no induced $P_4$
$abcd$ such that $a<b$ and $d<c$.
A graph is called {\em perfectly orderable} if there exists a perfect ordering $(v_1,\dots,v_n)$ of 
the vertex set of $G$.
For example, the following graphs are perfectly orderable:
\begin{itemize}
\item
comparability graphs of finite posets (including bipartite graphs);
\item
chordal graphs;
\item
complements of chordal graphs;
\item
weakly chordal graphs with no $P_5$ (\cite{Hayward1});
\item
bull-free graphs with no odd holes and no antiholes (\cite{Hayward2}).
\end{itemize}
For a bipartite graph $G$, the complement of $G$ is perfectly orderable if and only if  $G$ is chordal bipartite.
It is known \cite{perfectlycontractile} that every perfectly orderable graph is
perfectly contractile.
A graph $G$ is called {\em strongly perfect} if
every induced subgraph $H$ of $G$ has a stable set meeting all maximal cliques of $H$. 
It is known \cite[Theorem~2]{perfectlyorderable} that every perfectly orderable graph is strongly perfect.

Let $G$ be a perfectly orderable graph with a perfect ordering $(v_1,\dots,v_n)$.
Then we define the ordering of stable sets $(S_1,\dots, S_m)$ by
$$
\tb^{\rho(S_1)} >_{\rm lex} \dots >_{\rm lex} \tb^{\rho(S_m)},
$$
where $>_{\rm lex}$ is a lexicographic order on $K[\tb]$ induced by the ordering
$t_1 > \dots > t_n$.
Let $>_{\rm rev}$ be a reverse lexicographic order on $K[\xb]$ induced by the ordering
$x_1 < \dots < x_m$.
Then we have the following.

\begin{Theorem}
\label{thm:perfectorder}
Let $G$ be a perfectly orderable graph.
Then the initial ideal of $I_G$ with respect to 
the reverse lexicographic order $>_{\rm rev}$ defined above is squarefree and quadratic.
\end{Theorem}

\begin{proof}
Given two stable sets $S_i$ and $S_j$ with $1 \le i < j \le m$,
let $H_{ij}$ be the induced subgraph of $G$ induced by the vertex set $S_i \cup S_j$.
Note that $H_{ij}$ is bipartite.
Let $(v_{i_1}, \dots, v_{i_p})$ be the induced perfect ordering of $H_{ij}$.
Then we define the stable set $S$ ($\subset S_i \cup S_j$) as follows:
scan $(v_{i_1}, \dots, v_{i_p})$ from $v_{i_1}$ to $v_{i_p}$,
 and place each $v_{i_q}$ in $S$ if and only if none of its neighbors $v_{i_{q'}}$ ($q' < q$) has been placed in $S$.
In particular, we have $v_{i_1} \in S$.
In the proof of \cite[Theorem~2]{perfectlyorderable},
it is shown that $S$ is a stable set of $H_{ij}$ meeting all maximal cliques of $H_{ij}$
(remark that, since $H_{ij}$ is bipartite, maximal cliques of $H_{ij}$ are 
 edges and isolated vertices). 
Since $S$ is a stable set of $G$, $S=S_k$ for some $1 \le k \le m$.
We now show that $\rho(S_i) + \rho(S_j) - \rho(S_k) = \rho(S_\ell)$
for some $1 \le \ell \le m$.
Since $H_{ij}$ is perfect, by \cite[Theorem~3.1]{facet_perfect_graphs},
the stable set polytope $\Qc_{H_{ij}}$ of $H_{ij}$ is the set of all vectors
$(y_1,\dots, y_n)$ satisfying
\begin{eqnarray*}
y_i \ge 0, & \mbox{ for all } i \in [n]\\
\sum_{i \in W} y_i \le 1, & \mbox{ for any maximal clique } W \mbox{ of } H_{ij}.
\end{eqnarray*}
Let $(y_1,\dots, y_n) = \rho(S_i) + \rho(S_j)$.
Since $(y_1,\dots, y_n)$ belongs to $2 \Qc_{H_{ij}}$,
we have 
$$
\sum_{i \in W} y_i \le 2,
$$
for any maximal clique $W$ of $H_{ij}$.
Moreover, since $S_k$ meets all maximal cliques of $H_{ij}$,
$(y_1',\dots, y_n') = \rho(S_i) + \rho(S_j)  - \rho(S_k) $ satisfies the inequalities
$$
\sum_{i \in W} y_i' \le 1,
$$
for any maximal clique $W$ of $H_{ij}$.
In addition, $y_i' \ge 0$ follows from $S_k \subset S_i \cup S_j$.
Thus $(y_1',\dots, y_n') $ belongs to $\Qc_{H_{ij}} \cap \ZZ^n$.
Since $\Qc_{H_{ij}}$ is a $(0,1)$-polytope, $\Qc_{H_{ij}} \cap \ZZ^n$
coincides with the vertex set of $\Qc_{H_{ij}}$, and hence
 $(y_1',\dots, y_n') $ is equal to $\rho(S_\ell)$
for some stable set  $S_\ell$ of $H_{ij}$.
Then $f_{ij} := x_i x_j - x_k x_\ell$ belongs to $I_G$.
From the choice of $S_k$, we have $k \le i,j, \ell$.
Note that $f_{ij} = 0$ if and only if $i=k$.
If $i \ne k$, then the initial monomial of $f_{ij}$ is $x_i x_j$ with respect to $>_{\rm rev}$.

Let $\Gc = \{0 \neq f_{ij} \in K[\xb] : 1 \le i < j \le m\}$.
Suppose that $\Gc$ is not a Gr\"obner basis of $I_G$ with respect to $>_{\rm rev}$.
By \cite[Theorem~3.11]{BinomialIdeals}, there exists a binomial $f=u - v \in I_G$ such that
neither $u$ nor $v$ belong to ${\rm in}_{>_{\rm rev}}(\Gc)$.
Since $I_G$ is prime, we may assume that $f$ is irreducible.
Let $u = x_{i_1} \dots x_{i_r}$ and $v = x_{j_1} \dots x_{j_r}$,
with $i_1 \le \dots \le i_r$ and $j_1 \le \dots \le j_r$.
We may assume that $i_1 < j_1$.
Then the left-most nonzero component of $\rho(S_{i_1}) - \rho(S_{j_1}) $ is positive.
Suppose that the $\alpha$-th component of $\rho(S_{i_1}) - \rho(S_{j_1}) $
is the left-most nonzero component.
Since $f$ belongs to $I_G$, we have
$$
\sum_{k=1}^r \rho(S_{i_k})
=
\sum_{k=1}^r \rho(S_{j_k}).
$$
Then there exists $1 < k' \le r$ such that $v_\alpha \in S_{j_{k'}}$.
Suppose that $f_{j_1 j_{k'}} = 0$.
Then $S_{j_1}$ is the stable set  that meets all maximal cliques of $H_{j_1 j_{k'}}$
obtained by the above procedure.
The choice of $v_\alpha$ guarantees that there exist no neighbors of $v_\alpha$ in 
$\{v_1,\dots, v_{\alpha-1}\} \cap S_{j_1}$ since
$$
\{v_1,\dots, v_{\alpha-1}\} \cap S_{j_1}
=\{v_1,\dots, v_{\alpha-1}\} \cap S_{i_1}.
$$ 
Hence $\alpha \in S_{j_1} \cup S_{j_{k'}}$ should be included in $S_{j_1}$
in the above procedure.
This is a contradiction.
Thus we have $f_{j_1 j_{k'}} \neq 0$.
%the stable set defined above that meets all maximal cliques of $H_{j_1 j_{k'}}$ must contain $\alpha$.
This contradicts that $v$ does not belong to ${\rm in}_{>_{\rm rev}}(\Gc)$.
\end{proof}

\begin{Remark}
Theorem~\ref{thm:perfectorder} is a generalization of results on several toric ideals.
For example, the existence of a squarefree quadratic initial ideal for 
comparability graphs of posets \cite{HibiChainTriangulation},
and the complement of chordal bipartite graphs \cite{MOS} are guaranteed by this theorem.
Moreover, if $G$ is the complement of a bipartite graph $H$,
then the edge polytope \cite[Chapter~5]{BinomialIdeals} of $H$ is a face of $\Qc_G$.
Hence the existence of a squarefree quadratic initial ideal for the toric ideals of 
edge polytopes of chordal bipartite graphs \cite{OHkb} is a corollary of Theorem~\ref{thm:perfectorder}.
\end{Remark}

From Proposition \ref{clique_sum} and Theorem~\ref{thm:perfectorder},
we have the following.

\begin{Corollary}
\label{clique_sum_p.o.}
Suppose that a graph $G$ is obtained recursively
by gluing along cliques starting from perfectly orderable graphs.
Then there exists a monomial order such that
the initial ideal of $I_G$ is squarefree and quadratic.
\end{Corollary}

A graph $G$ is called {\em clique separable} if $G$ is obtained by
successive gluing along cliques starting with graphs of Type 1 or 2:

\smallskip

Type 1:
the join of a bipartite graph with more than 3 vertices with
a complete graph; 

Type 2:
a complete multipartite graph.

\smallskip

\noindent
Bertschi \cite{perfectlycontractile} proved that every clique separable graph
is perfectly contractile.

\begin{Theorem}
\label{thm:cliquesep}
Let $G$ be a clique separable graph.
Then there exists a monomial order such that
the initial ideal of $I_G$ is squarefree and quadratic.
\end{Theorem}

\begin{proof}
Suppose that $H$ is a graph of Type 1,
that is, $H$ is the join of a bipartite graph $H_1$ 
with more than 3 vertices with
a complete graph $H_2$. 
Since $H_1$ is bipartite, $H_1$ is perfectly orderable, 
and hence the vertex set of $H_1$ admits a perfect order.
Note that any vertex of $H_2$ is adjacent to any other vertex of $H$.
Thus no vertices of $H_2$ appear in an induced path $P_4$ in $H$.
Hence we can extend this order to a perfect order for $H$.
If a graph $H$ is of Type 2, then $H$ is the complement of
a disjoint union of complete graphs.
Since a disjoint union of complete graphs is a chordal graph,
$H$ is perfectly orderable.
Therefore any clique separable graph $G$ is 
obtained recursively
by gluing along cliques starting from perfectly orderable graphs.
Thus the desired conclusion follows from Corollary~\ref{clique_sum_p.o.}.
\end{proof}

Let $G(n)$ denote the set of all graphs on the vertex set $[n]$.
Given a property $P$, let $G_P(n)$ denote the set of all graphs on the vertex set $[n]$
with $P$.
Then we say that {\em almost all graphs} have the property $P$
if $\lim_{n \rightarrow \infty} |G_P(n)|/|G(n)|$ exists and is equal to $1$.
A graph $G$ is called {\em generalized split} if either $G$ or $\overline{G}$
satisfies the following:
\begin{itemize}
\item[$(*)$]
The vertex set of the graph is partitioned into disjoint cliques $C_0,C_1,\dots,C_k$
such that there are no edges between $C_i$ and $C_j$ for any $1 \le i < j \le k$.
\end{itemize}
Pr\"omel and Steger \cite{PrSt} proved that 
(i) generalized split graphs are perfect, and
(ii) almost all perfect graphs are generalized split.

\begin{Proposition}
\label{almostall}
Let $G$ be a generalized split graph with no antiholes.
Then each connected component of $G$ is obtained recursively
by gluing along cliques starting from perfectly orderable graphs.
In particular, there exists a monomial order such that
the initial ideal of $I_G$ is squarefree and quadratic.
\end{Proposition}

\begin{proof}
We may assume that $G$ is connected.
If $\overline{G}$ satisfies condition $(*)$, then 
it is shown in \cite[Proof of Corollary~2]{BlEn} that $G$ is perfectly orderable.
Suppose that $G$ satisfies condition $(*)$ and has no antiholes.
Then $G$ is obtained recursively
by gluing along cliques starting from $G_1, \dots, G_k$, where
$G_i$ is the induced subgraph of $G$ on the vertex set $V(C_0) \cup V(C_i)$
for $1 \le i \le k$.
Since the complement of $G_i$ is a bipartite graph and $G_i$ has no antiholes,
we have that the complement of $G_i$ is chordal bipartite.
Hence $G_i$ is perfectly orderable.
Thus $G$ is obtained recursively
by gluing along cliques starting from perfectly orderable graphs.
Therefore the desired conclusion follows from Corollary~\ref{clique_sum_p.o.}. 
\end{proof}

\begin{Remark}
From Theorem ~\ref{hitsuyou} and Proposition~\ref{almostall},
it follows that any generalized split graph with no antiholes 
has no odd stretchers as induced subgraphs.
\end{Remark}

\bibliographystyle{plain}
\bibliography{Bibliography.bib}

\end{document}

%% file: intro.tex
\section*{Introduction}
A graph $G$ is {\em perfect} if every induced subgraph $H$ of $G$ satisfies $\chi(H)=\omega(H)$, where $\chi(H)$ is the chromatic number of $H$ and $\omega(H)$ is the maximum cardinality of cliques of $H$.
Perfect graphs were introduced by Berge in \cite{perfect}.
A {\em hole} is an induced cycle of length $\ge 5$
and an {\em antihole} is the complement of a hole.
In 2006, Chudnovsky, Robertson, Seymour and Thomas solved a famous conjecture in graph theory, which was conjectured by Berge and is now known as the Strong Perfect Graph Theorem:

\smallskip

\noindent
{\bf The Strong Perfect Graph Theorem}
(\cite{perfecttheorem}){\bf .}
A graph is perfect if and only if it has no odd holes and no odd antiholes as induced subgraphs.

\smallskip

On the other hand, Bertschi introduced a hereditary class of perfect graphs in \cite{perfectlycontractile}.
An {\em even pair} in a graph $G$
is a pair of non-adjacent vertices of $G$ such that
the length of all chordless paths between them is even.
Contracting a pair of vertices $\{x, y\}$ in a graph $G$ means removing $x$ and
$y$ and adding a new vertex $z$ with edges to every neighbor of $x$ or $y$.
A graph $G$ is called {\em even-contractile}
if there is a sequence $G_0, \dots, G_k$ of graphs such that $G=G_0$, each $G_i$
 is obtained from $G_{i-1}$ by contracting an even pair of  $G_{i-1}$, and $G_k$
is a clique.
A graph is called {\em perfectly contractile} if
all induced subgraphs of $G$ are even-contractile.
Every perfectly contractile graph is perfect.
Moreover, the following graphs are perfectly contractile (\cite{perfectlycontractile});
\begin{itemize}
	\item Meyniel graphs;
	\item perfectly orderable graphs;
	\item clique separable graphs.
\end{itemize}
On the other hand, a necessary condition for a graph to be perfectly contractile 
is known.
An {\em odd stretcher} (or {\em prism}) graph $G_{s,t,u}$
($1 \le s,t,u \in \ZZ$) is a graph 
on the vertex set \[\{i_1,i_2,\dots,i_{2s}, j_1,j_2,\dots,j_{2t},k_1,k_2,\dots,k_{2u}\}\]
with edges
$$
\{i_1,j_1\},
\{i_1,k_1\},
\{j_1,k_1\},
\{i_{2s},j_{2t}\},
\{i_{2s},k_{2u}\},
\{j_{2t},k_{2u}\},
%$$
%$$
\ \ 
\{i_1,i_2\},\{i_2,i_3\},\dots,\{i_{2s-1},i_{2s}\},
$$
$$
\{j_1,j_2\},\{j_2,j_3\},\dots,\{j_{2t-1},j_{2t}\},
%$$
%$$
\{k_1,k_2\},\{k_2,k_3\},\dots,\{k_{2u-1},k_{2u}\}.
$$
If $s=t=u=1$, then the graph $G_{1,1,1}$ coincides with 
the antihole of length 6.
In \cite{planarperfectly}, it was formally shown that if a graph $G$ is perfectly contractile, then $G$ has no odd holes, no (odd/even) antiholes and no odd stretchers.
It is conjectured \cite{pc_conjecture} that this necessary condition is also a sufficient condition for perfectly contractile graphs.
Namely,
\begin{Conjecture}
\label{graphconjecture}
A graph $G$ is perfectly contractile
if and only if $G$ contains no odd holes, no (odd/even) antiholes and no odd stretchers as induced subgraphs.
\end{Conjecture}

By the Strong Perfect Graph Theorem, 
 a graph $G$ contains no odd holes, no antiholes and no odd stretchers as induced subgraphs
if and only if $G$ is perfect and contains no even antiholes and no odd stretchers as induced subgraphs.
Conjecture \ref{graphconjecture} is true for 
\begin{itemize}
\item
planar graphs (\cite{planarperfectly}),
\item
dart-free graphs (\cite{dartfree}) (including claw-free graphs),
\item
even stretcher-free graphs (\cite{prismfree}) (including bull-free graphs (\cite{bullfree})).
\end{itemize}
However, Conjecture \ref{graphconjecture} is still open.

We consider this conjecture from a viewpoint of commutative algebra.
In particular, we study the toric ideal of the stable set polytope of a graph.
Let $G$ be a finite simple graph on the vertex set $[n]:=\{1,2,\dots,n\}$ and the edge set $E(G)$.
A subset $S \subset [n]$ is called a {\em stable set} (or {\em independent set}) of $G$
if $\{i,j\} \notin E(G)$ for all $i,j \in S$ with $i \neq j$.
In particular, the empty set $\emptyset$ and any singleton $\{i\}$ with $i \in [n]$
are stable.
Let $S(G)=\{S_1,\ldots,S_m\}$ denote the set of all stable sets of $G$.
Given a subset $S \subset [n]$, we associate the $(0,1)$-vector $\rho(S) = \sum_{j \in S} \eb_j$.
Here $\eb_j$ is the $j$-th unit coordinate vector in $\RR^n$.
For example, $\rho(\emptyset) = {\bf 0} \in \RR^n$.
Then the {\em stable set polytope} $\Qc_G \subset \RR^n$ of a simple graph $G$ is the convex hull of 
$\{ \rho(S_1),\ldots,\rho(S_m)\}$.
Now, we define the toric ideal of the stable set polytope of a graph.
Let $K[\xb]:=K[x_1,\ldots,x_m]$ and $K[\tb,s]:=K[t_1,\ldots,t_n,s]$ be polynomial rings over a field $K$. Then the {\em toric ideal} of the stable set polytope $\Qc_G$ of $G$ is the kernel of a homomorphism $\pi_G : K[\xb] \to K[\tb,s]$ defined by $\pi_G(x_i)=\tb^{\rho (S_i)}s$. Here, for a nonnegative integer vector $\ab=(a_1,\ldots,a_n) \in \ZZ^n$, we denote $\tb^{\ab}=t_1^{a_1} \cdots t_n^{a_n} \in K[\tb,s]$. On the other hand, the image of $\pi_G$ is denoted by $K[G]$ and called the {\em toric ring} of $\Qc_G$.
The toric ring $K[G]$ is called {\em quadratic} if $I_G$ is generated by quadratic binomials.

Let $\Mc_m$ be the set of all monomials in $K[\xb]$.
A total order $<$ on $\Mc_m$ is called a {\em monomial order}
if $<$ satisfies (i) $1 \in \Mc_m$ is the smallest monomial in $\Mc_m$;
(ii) if $u, v, w \in \Mc_m$ satisfies $u<v$, then we have $uw < vw$.
The {\em initial monomial} ${\rm in}_<(f)$ of  a nonzero
polynomial $f \in K[\xb]$ is the largest monomial appearing in $f$.
Given an ideal $I \subset K[\xb]$, 
the {\em initial ideal} ${\rm in}_< (I)$ of $I$ is a monomial ideal generated by $\{{\rm in}_<(f) : 0 \neq f \in I\}$.
The initial ideal is called {\em squarefree} (resp. {\em quadratic}) if it is generated by squarefree (resp. quadratic) monomials.
A finite subset $\{g_1, \dots, g_t\} \subset I$ is called a
{\em Gr\"obner basis} of $I$ with respect to $<$ if ${\rm in}_< (I)$ is generated by
$\{{\rm in}_<(g_1) ,\dots, {\rm in}_<(g_t)\}$.
If $\{g_1, \dots, g_t\}$ is a Gr\"obner basis of $I$,
then $\{g_1, \dots, g_t\}$ generates $I$.
See, e.g., \cite{BinomialIdeals} for details on toric ideals and their Gr\"obner bases in general.

We can characterize when a graph $G$ is perfect in terms of $I_G$.
In fact, a graph $G$ is perfect if and only if the initial ideal of $I_G$ is squarefree with respect to any reverse lexicographic order (\cite{GPT, OHcompressed, Sul}). Similarly, we can characterize when a graph is perfect in terms of the Gorenstein property and the normality of toric rings associated to finite simple graphs (see \cite{HOTperfect,HTperfect,HTorderperfect,OHperfect}).

In the present paper, for a perfect graph $G$, we discuss when the toric ideal $I_G$ is generated by quadratic binomials. 
First, we conjecture the following:
\begin{Conjecture}
\label{toricconjecture}
	Let $G$ be a perfect graph. Then 
		$I_G$ is generated by quadratic binomials if and only if $G$ contains no odd holes, no (odd/even) antiholes and no odd stretchers as induced subgraphs.
\end{Conjecture}

The ``only if'' part of Conjecture \ref{toricconjecture} is true (Proposition~\ref{even_antihole} and Theorem~\ref{hitsuyou}).
On the other hand, it is known that $I_G$ has a squarefree quadratic initial ideal, and in particular, $I_G$ is generated by quadratic binomials if $G$ is either the comparability graph of a poset (\cite{HibiChainTriangulation}), an almost bipartite graph (\cite[Theorem~8.1]{EnNo}), a chordal graph or a ring graph (\cite{EnNo, MOS}), the complement of a chordal bipartite graph (\cite[Corollary~1]{MOS}).

We consider Conjecture \ref{toricconjecture} for some classes of perfectly contractile graphs.
In fact, we will prove that Conjecture \ref{toricconjecture} is true for the following graphs:
\begin{itemize}
	\item Meyniel graphs (Theorem \ref{thm:meyniel});
	\item perfectly orderable graphs (Theorem \ref{thm:perfectorder});
	\item clique separable graphs (Theorem \ref{thm:cliquesep}).
\end{itemize}
In particular, we will show that  if $G$ is either a perfectly orderable graph or a clique separable graph, then $I_G$ possesses a squarefree quadratic initial ideal.
We also remark that Conjecture \ref{toricconjecture} is true for generalized split graphs
(Proposition~\ref{almostall}).
It is known \cite{PrSt} that ``almost all'' perfect graphs are generalized split.